\documentclass[10pt,leqno,twoside]{amsart}
\usepackage{enumerate}
\usepackage{amssymb}
\usepackage{amssymb,amsthm,amsmath,eucal,mathrsfs}
\usepackage{comment}
\usepackage{tikz, subfigure, xcolor} 

\usepackage{pgfplots}

\usepackage{cite}

\usepackage{amsmath}

\usepackage{amsthm}

\usepackage{amssymb}

\usepackage{amsfonts}

\usepackage{caption} 

\usepackage{dsfont}

\usepackage{hyperref, doi}

\newtheorem{theorem}{Theorem}[section]
\newtheorem{lemma}[theorem]{Lemma}
\newtheorem{proposition}[theorem]{Proposition}

\theoremstyle{definition}

\newtheorem{remark}[theorem]{Remark}

\numberwithin{equation}{section}

\renewcommand{\div}{\mathrm{div}\,}    

\newcommand{\R}{\mathbb{R}}

\newcommand{\T}{\mathbb{T}}

\newcommand{\C}{\mathbb{C}}

\newcommand{\Z}{\mathbb{Z}}

\newcommand{\N}{\mathbb{N}}

\newcommand{\PP}{\mathbb{P}}

\newcommand{\ve}{\varepsilon}

\newcommand\restr[2]{{
  \left.\kern-\nulldelimiterspace 
  #1 
  \vphantom{\big|} 
  \right|_{#2} 
  }}

\usepackage{eucal}

\newcommand{\divh}{\mbox{div}_H \;}

\newcommand{\vertiii}[1]{
{\left\vert\kern-0.25ex\left\vert\kern-0.25ex\left\vert #1 
    \right\vert\kern-0.25ex\right\vert\kern-0.25ex\right\vert}
}

\title[The Primitive Equations in the scaling invariant space $L^{\infty}(L^1)$]{The Primitive Equations in the scaling invariant space $L^{\infty}(L^1)$}

\subjclass[2010]{Primary: 35Q35; Secondary: 76D03, 47D06, 86A05.}

\keywords{primitive equations, rough data, global string well-posedness \\
This work was partly supported by the DFG International Research Training Group IRTG 1529 and the JSPS Japanese-German Graduate Externship on Mathematical Fluid Dynamics. 
The first author is partly supported by JSPS through grant Kiban S (No. 26220702), Kiban A (No. 17H01091), Kiban B (No. 16H03948) and the second and fourth author are supported by IRTG 1529 at TU Darmstadt.}

\author{Yoshikazu Giga} 
\address{Graduate School of Mathematical Sciences, University of Tokyo, Komaba 3-8-1, Meguro-ku, Tokyo, 153-8914, Japan }
\email{labgiga@ms.u-tokyo.ac.jp}

\author{Mathis Gries} 
\address{Departement of Mathematics,
	TU Darmstadt, Schlossgartenstr. 7, 64289 Darmstadt, Germany}
\email{gries@mathematik.tu-darmstadt.de}

\author{Matthias Hieber} 
\address{Departement of Mathematics,
	TU Darmstadt, Schlossgartenstr. 7, 64289 Darmstadt, Germany}
\email{hieber@mathematik.tu-darmstadt.de}

\author{Amru Hussein} 
\address{Departement of Mathematics,
	TU Darmstadt, Schlossgartenstr. 7, 64289 Darmstadt, Germany}
\email{hussein@mathematik.tu-darmstadt.de}

\author{Takahito Kashiwabara}
\address{Graduate School of Mathematical Sciences, The University of Tokyo, 3-8-1 Komaba, Meguro, Tokyo 153-8914, Japan}
\email{tkashiwa@ms.u-tokyo.ac.jp}

\begin{document}
\begin{abstract}
Consider the primitive equations on $\R^2\times (z_0,z_1)$ with initial data $a$ of the form $a=a_1+a_2$, where $a_1 \in BUC_\sigma(\R^2;L^1(z_0,z_1))$ and 
$a_2 \in L^\infty_\sigma(\R^2;L^1(z_0,z_1))$ and where $BUC_\sigma(L^1)$ and $L^\infty_\sigma(L^1)$ denote the space of all solenoidal, bounded uniformly continuous 
and all solenoidal, bounded functions on $\R^2$, respectively, which take  values in $L^1(z_0,z_1)$. These spaces are scaling invariant and represent the anisotropic character of these equations. 
It is shown that, if $\|a_2\|_{L^\infty_\sigma(L^1)}$ is sufficiently small, then this set of equations has a unique, local, mild solution. 
If in addition $a$ is periodic in the horizontal variables,
then this solution is a strong one and extends to a unique, global, strong solution. The primitive equations are thus strongly and globally well-posed for these data.    
The approach depends crucially on mapping properties of the hydrostatic Stokes semigroup in the $L^\infty(L^1)$-setting and can thus be  seen as the counterpart of the 
classical iteration schemes for the Navier-Stokes equations for the situation of the primitive equations. 
\end{abstract}

\maketitle

\section{Introduction}\label{sec:intro}

The primitive equations for ocean and atmospheric dynamics 
serve 
as a fundamental model for many geophysical flows. This set of equations describing the conservation of momentum and mass of a fluid, assuming 
hydrostatic balance of the pressure, is given in the isothermal setting by
\begin{align}\label{eq:prim}
\left\{
\begin{array}{rll}
\partial_t v + u \cdot \nabla v  - \Delta v + \nabla_H \pi  & = 0, \quad &\text{ in } \Omega \times (0,T),  \\
\partial_z \pi & = 0, \quad &\text{ in } \Omega \times (0,T),  \\
\div u & = 0, \quad &\text{ in } \Omega \times (0,T),\\
v(0)   & = a.
\end{array}\right.
\end{align}
Here $\Omega:= \R^2 \times J$, where $J =(z_0,z_1)$ is an interval. Denoting the horizontal coordinates by $x,y\in \R^2$ and the vertical one by $z\in (z_1,z_2)$, we use the notation 
$\nabla_H = \left(\partial_x, \partial_y \right)^T$, whereas $\Delta$ denotes the three dimensional Laplacian and $\nabla$ and $\div$ the three dimensional gradient and divergence operators.
The velocity $u$ of the fluid is described by  $u=(v,w)$, where $v=(v_1,v_2)$ denotes the horizontal component and $w$ the vertical one.

In the literature various sets of boundary conditions are considered such as Neumann, Dirichlet and mixed boundary conditions. In this article 
we 
choose 
Neumann boundary conditions for $v$, i.e.   
\begin{align}\label{eq:bc}
\left\{
\begin{array}{rrrr}
\frac{\partial }{\partial z} v & =0, \quad \text{ on } \partial\Omega \times (0,T),\\
w &= 0, \quad \text{ on } \partial\Omega \times (0,T).
\end{array}\right.
\end{align}    

The rigorous analysis of the primitive equations started with the pioneering work of Lions, Temam and Wang \cite{Lionsetall1992,Lionsetall1992_b, Lionsetall1993}, who proved the 
existence of a global weak solution for this set of equations for initial data $a \in L^2$.
The uniquness problem for weak solutions remains an open problem until today.


The existence of a local, strong solution for this equation with data $a \in H^1$ was proved by Guill\'en-Gonz\'alez, Masmoudi and Rodiguez-Bellido in \cite{Guillenetall2001}. 

In 2007, Cao and Titi \cite{CT07} proved a breakthrough result for this set of equation which says, roughly speaking,  that there exists a unique, global strong solution to 
the primitive equations for {\em arbitrarily large} initial data $a\in H^1$. Their proof is based on {\em a priori} $H^1$-bounds for the solution, which in turn are  
obtained by $L^\infty(L^6)$ energy estimates.
Kukavica and Ziane \cite{KuZi07}   
proved global strong well-posedness of the primitive equations with respect to arbitrary large $H^1$-data for the case of mixed Dirichlet-Neumann boundary conditions. For a different approach see also Kobelkov \cite{Kobelkov2007}. 
%
%
%
We also would like to draw the attention of the reader to the recent survey article by Li and Titi \cite{LiTiti2016} on the primitive equations. 

Recently, an new approach to the primitive equations based on the theory of evolution equations has been developed in \cite{HiKa16} and \cite{HieberHusseinKashiwabara2016}.
This approach is also valid in the $L^p$-setting for all $1<p<\infty$ and, using this approach, the authors  proved global strong well-posedness of the primitive equations for 
arbitrary large data in $H^{2/p,p}$ subject to 
mixed Dirichlet-Neumann
boundary conditions.
Taking formally the limit $p\to \infty$, it is now tempting to consider initial data $a \in L^{\infty}$ with no differentiability assumption on the initial data. This article 
aims to find a function space, as large as possible, for the initial data for which the primitive equations are strongly and globally well-posed.  

Recent regularity results for weak solutions by Li and Titi \cite{LiTi17} and Kukavica,
Pei, Rusin and Ziane \cite{KPRZ14} are also pointing in this direction. More specificially, starting from a weak solution to the primitive equations, these authors investigated 
regularity criteria for weak solutions for the primitive equations, following hereby in a certain sense the  spirit of Serrin's condition in the theory of the Navier-Stokes equations
and methods of weak-strong uniqueness. Li and Titi proved in \cite{LiTi17} that weak solutions are unique for initial
values in $C^0$ or in $\{u\in L^6\colon \partial_z u \in L^2\}$ including a small perturbation belonging to $L^{\infty}$. By the weak-strong uniqueness property, it follows that these 
weak solutions regularize and  become strong solutions. 

Our approach to rough initial data results for the primitive equations is very different: it considers the primitive equation as an evolution equation in an anisotropic 
function space of the form $L^{\infty}(\R^2;L^1(J))$. This space is invariant under the scaling
\begin{align*}
v_{\lambda}(t,x_1,x_2,x_3)=\lambda v(\lambda^2 t, \lambda (x_1,x_2,x_3)),  \quad \lambda >0. 
\end{align*}
By this we mean that
$\|v_{\lambda}\|_{L^{\infty}(\R^2;L^1(\lambda^{-1} J))} = \|v\|_{L^{\infty}(\R^2;L^1(J))}$ for all $\lambda >0$. Moreover, $v_{\lambda}$ is a solution to the primitive equations whenever $v$ has this property. For further information on the Navier-Stokes equations in critical spaces see \cite{Gilles, Chemin, GigaSaal}.

Based on $L^\infty$-type estimates for the underlying hydrostatic Stokes semigroup and as well as on  its  
gradient estimates, we develop an iteration scheme yielding first the existence of a unique, local mild solution for initial data of the form 
$a=a_1+a_2$ with 
\begin{align*}
a_1\in BUC_\sigma(\R^2,L^1(J)) \mbox{ and } a_2 \mbox{ being a small perturbation in } L^{\infty}_\sigma(\R^2;L^1(J)).
\end{align*}
Here $BUC(\R^2)$ denotes the space of all bounded and uniformly continuous functions on $\R^2$. 
The subscript $\sigma$ means the subspace of 
solenoidal fields rigorously defined in the next section.
Assuming that $a_1,a_2$ are periodic with respect to horizontal variables, we are able to prove that 
the solution regularizes sufficiently and thus, by an appropriate a priori estimate,  can be extended to global, strong solution without any restriction on the size of $a_1$.

Comparing our assumptions on the initial data with the ones given by Li and Titi \cite{LiTi17}, observe first
that our assumptions are slightly less restrictive compared to the case of continuous initial data, while our assumptions are not comparable to their second case. 


Our approach may be  viewed as the counterpart of the classical iteration schemes for the Navier-Stokes equations due to Giga  \cite{Gig86} and Kato \cite{Kat84}. Note that, in 
contrast to  the case 
of the Navier-Stokes equations, our iteration scheme presented here combined with a suitable a priori estimate yields the existence  of a unique, global strong solutions not only for small data 
as in the case of the Navier-Stokes equations, but for  arbitrary large solenoidal data $a \in BUC_\sigma(\R^2;L^1(J))$.

As written above, our approach depends crucially on $L^\infty(L^1)$-mappig properties of the underlying hydrostatic Stokes semigroup, 
including 
gradient estimates. 
These are
collected in Proposition~\ref{thm4.5} and are of independent interest.  

This article is organized as follows. Section~\ref{sec:premain} presents the two main  results of this article. The following Sections \ref{sec:Interpolation}, \ref{sec:anisotropic} and \ref{sec:hydrostaticStokes} are devoted to anisotropic estimates for fractional 
derivatives, the heat semigroup as well as for the hydrostatic semigroup. In Section~\ref{sec:proofs} we present a proof of our main results based on an iteration scheme.    

\vspace{1.5cm}

\section{Preliminaries and main results}\label{sec:premain} 

Let  $z_0 \in \R$, $z_1=z_0+h$ for some $h>0$ and $J$ be the interval  $J=(z_0,z_1)$. 
The incompressibility condition $\mbox{div } u =0$ in $\Omega \times (0,T)$ implies 
$$
w(x,y,z)= \int_{z}^{z_1} \divh v(x,y,z)dz, 
$$ 
where the boundary condition $w=0$ on $\partial\Omega$ was taken into account. Also, $w=0$ on $\partial\Omega$ implies 
$$
\divh \overline{v} = 0 \mbox{ in } \R^2,
$$
where $\overline{v}$ denotes the vertical average of $v$, i.e. 
$$
\overline{v}(x,y) = \frac{1}{z_1-z_0} \int_{z_0}^{z_1} v(x,y,z) dz.
$$
The linearization of equation \eqref{eq:prim} are the \textit{hydrostatic Stokes equations}, which are given by 
\begin{align}\label{eq:hydrostokes}
\left\{
\begin{array}{rll}
\partial_t v -\Delta v +\nabla_H \pi = 0,   &\text{ in } \Omega \times (0,T),  \\
\div_H \overline{v} = 0, &\text{ in } \Omega \times (0,T),  \\
v(0)= a &\text{ in } \Omega.  \\
\end{array}\right.
\end{align}
The name 'hydrostatic Stokes equations' is motivated by the assumption of a hydrostatic balance when deriving the full primitive equations. The equations \eqref{eq:hydrostokes} are 
supplemented by Neumann boundary conditions \eqref{eq:bc} for $v$. 

For a function $f: \R^2 \times J \to \C$,  we define for $1 \leq p,q < \infty$ the $L^q(\R^2,L^p(J))$-norm of $f$ by
$$
\|f\|_{L^q(\R^2;L^p(J))} := \Big( \int_{\mathbb{R}^2} \Big(\int_J\left|f(x',x_3)\right|^q \mathrm{d}x_3 \Big)^{q/p} \mathrm{d}x' \Big)^{1/q},
$$
where we use the shorthand notation $L^q(L^p)$ for the spaces and $\|\cdot\|_{q,p}$ for the norms. 
The usual modifications hold for the cases $p=\infty$ or $q=\infty$. 
The space  $L^p(\R^2;L^q(J))$ consisting of all measurable function $f$ with $\|f\|_{p,q}<\infty$ and equipped with the above norm becomes a Banach space. 

Following \cite{HiKa16}, we introduce the hydrostatic Helmholtz projection as follows.  For a function $f=(f_1,f_2):\mathbb{R}^2\times J \to \C$, we define the hydrostatic 
Helmholtz projection by
$$ 
\PP f := f  + \nabla_H (-\Delta)^{-1}\divh \overline{f}. 
$$
The solenoidal subspace $L^\infty_\sigma(\R^2;L^p(J))$ is then defined for $1\leq p \leq \infty$ as the closed subspace of $L^\infty(\R^2;L^p(J))$ given by  
$$
L^\infty_\sigma(\R^2;L^p(J)):= \{v \in L^\infty(\R^2;L^p(J)):\int_{\R^2} \overline{v} \; \nabla_H \varphi \; dx = 0 \mbox{ for all } \varphi \in \widehat{W}^{1,1}(\R^2)\},
$$
where $\widehat{W}^{1,1}(\R^2)$ denotes the homogeneous Sobolev space of the form $\widehat{W}^{1,1}(\R^2)=\{\varphi \in L^1_{loc}(\R^2): \nabla_H \varphi \in L^1(\R^2)\}$, 
that is, the condition $\divh \overline{v} =0$ is understood in the sense of distributions.

If  $a \in L^\infty_\sigma(\R^2;L^p(J))$ for some $1\leq p \leq \infty$, then the solution of  equation \eqref{eq:hydrostokes} can be represented as $v(t)=S(t)a$, where 
$S$ denotes the {\em hydrostatic Stokes semigroup} on $L^\infty_\sigma(\R^2;L^p(J))$. The latter  semigroup may be represented as follows:   
consider the one-dimensional heat equation 
		\begin{align*}
		u_t - u_{zz} = 0, \quad  u(0)= u_0, 	
\end{align*}
in $J\times(0,\infty)$, where  $J=(z_0,z_1)$, subject to  the boundary conditions
		\begin{subequations}
		\begin{equation}
		u_z(z_1)=0,\ u_z(z_0)=0 \quad \mbox{or}
		\label{3.2a} 
		\end{equation}
		\begin{equation}
		u_z(z_1)=0,\ u(z_0)=0.
		\label{3.2b} 
		\end{equation}
		\end{subequations}
Given $u_0 \in L^p(J)$ for some $p \in [1,\infty]$, its solution $u$ is given by $u(t)=e^{t\Delta_N}u_0$ for \eqref{3.2a} and by  $e^{t\Delta_{DN}}u_0$ for \eqref{3.2b}, 
where $e^{t\Delta_N}$  and $e^{t\Delta_{DN}}$ denote the analytic semigroups on $L^p(J)$ generated by the Laplacian subject to Neumann or Dirichlet-Neumann boundary conditions, respectively.  
For $a \in L^\infty_\sigma(\R^2;L^p(J))$, the solution  of \eqref{eq:hydrostokes} is thus given by $v(t)=S(t)a$, where 
\[
S(t) = e^{t\Delta_H} \otimes e^{t\Delta_N}, \quad t>0,
\]
and where $e^{t\Delta_H}$ denotes the heat semigroup on $L^\infty(\R^2)$. The semigroup $S$ is not strongly continuous on $L^\infty_\sigma(\R^2;L^p(J))$, however, its restriction to  
\begin{align*}
BUC_\sigma(L^p)&:= BUC(\R^2;L^p(J)) \cap L^\infty_\sigma(L^p)
\end{align*}
defines for $1 \leq p <\infty$ a bounded analytic $C_0$-semigroup on this space.

Our first main result concerns the existence of a unique, mild solution to the primitive equations 
with initial value $a$, i.e. a function $v$ satisfying
		\begin{equation}
		v(t) = S(t)a - \int^t_0 S(t-s)\PP\nabla\cdot(u(s) \otimes v(s))\mathrm{d}s, \quad 0 < t < T,
		\label{5.1}
		\end{equation}
where $w(s)=\int^{z_1}_z \operatorname{div}_H v(s)\; \mathrm{d}x_3$ and $u(s)=(v(s),w(s))$ for all $s \in [0,t]$, and $\nabla\cdot(u \otimes v) =u \cdot \nabla v$ since $\div u =0$.

\vspace{.2cm}\noindent
\begin{theorem}\label{thm1.1}
Assume that $a$ is of the form $a=a_1+a_2$ with $a_1 \in BUC_\sigma(L^1)$ and $a_2 \in L^\infty_\sigma(L^1)$. Then there exists a constant $\varepsilon_0>0$ such that 
if $\|a_2\|_{L^\infty_\sigma(L^1)}<\varepsilon_0$, there exists $T>0$ such that \eqref{eq:prim} subject to \eqref{eq:bc} admits a unique, local mild solution $v$ with
		\[
		v \in C \left((0,T);BUC_\sigma\left(L^1(J)\right)\right)
		\]
satisfying 
		\[
		\overline{\lim_{t\to 0}}\; t^{1/2} \|\nabla v(t)\|_{L^\infty(L^1)} 
		\leq C\|a_2\|_{L^\infty_\sigma(L^1)}
		\]
for some constant  $C>0$ independent of $a_1$ and $a_2$.
\end{theorem}
In particular, one has $v-Sa_2 \in C \left([0,T);BUC_\sigma\left(L^1(J)\right)\right)$, where $Sa_2 \in C \left((0,T);BUC_\sigma\left(L^1(J)\right)\right)\cap L^{\infty} \left(0,T;BUC_\sigma\left(L^1(J)\right)\right)$.
%
The mild solution constructed in Theorem \ref{thm1.1} exists at least for some nontrivial time interval $[0,T)$ where $T$ depends on $a$. 
We are further able to estimate the existence time $T>0$ from below in terms of the following $\vertiii{\cdot}$-norm defined for $a \in L^\infty_\sigma(\R^2;L^1(J))$ by  
$$
\vertiii{a} :=[a]_\mu+\|a\|_{\infty,1}, \quad \mbox{where} \quad [a]_\mu := \sup_{0<t<1} t^\mu \left\| \nabla S(t)a \right\|_{\infty,1}, 
$$
for some $\mu \in [0,1/2)$

\vspace{.2cm}\noindent
\begin{proposition}[Estimate on life span]\label{cor2}\mbox{}\\
	Assume that $a \in BUC_\sigma(L^1)$, 
	i.e. $a_2=0$,
	satisfies $[a]_\mu<\infty$ for some $\mu\in[0,1/2)$. Then there exists a unique, local mild solution 
	$v \in C\left([0,T],BUC_\sigma(L^1)\right)$ of \eqref{eq:prim} and \eqref{eq:bc}. Moreover, there exists $c_*$, depending only on $\mu$, such that 
	\[
	1/T \leq  \min \left(c_*\vertiii{a},1\right)^{2/(1/2-\mu)}
	\]
	This estimate for $1/T$ remains  valid for $a \in L^\infty_\sigma(L^1)$ if $C\left([0,T),BUC_\sigma(L^1)\right)$ is being replaced by $L^\infty((0,T),L^\infty_\sigma(L^1))$.
\end{proposition}




Using $L^\infty({L^p})$-$L^\infty({L^1})$ smoothing properties of the semigroup and assuming initial data in $L^\infty({L^p})$ for $p>1$, one can control the $L^\infty({L^p})$-norms of $v(t), \nabla v(t)$ by the corresponding $L^\infty({L^1})$-norms thus improving the results of Theorem~\ref{thm1.1}.  

\vspace{.2cm}\noindent 
\begin{proposition}[Local existence for $p>1$]\label{cor1} \mbox{}\\
Let $T>0$ be as in Theorem \ref{thm1.1}. If in addition to the assumptions of Theorem \ref{thm1.1} the initial data $a$ satisfies
\begin{itemize}
\item[a)] $a \in L^\infty_\sigma(L^p)$ for some  $p \in (1,\infty]$, then 
$t^{1/2-1/2p} v, t^{1-1/2p}\nabla v \in L^\infty\left(0,T; L^\infty_\sigma(L^p)\right)$;
\item[b)] $a \in BUC_\sigma(L^p)$ for some  $p \in (1,\infty]$, then 
$t^{1/2-1/2p} v, t^{1-1/2p}\nabla v \in C\left([0,T),BUC_\sigma(L^p) \right)$
\item[c)] $a \in BUC_\sigma(BUC)$, then $t^{1/2}v, t\nabla v \in C\left( [0,T),BUC_\sigma(BUC) \right)$.
\end{itemize}
\end{proposition}

The following theorem on the global strong well-posedness of the primitive  equation with rough and arbitrary data is the second main result of this article.   

\vspace{.2cm}\noindent
\begin{theorem}[Global existence]\label{thm1.2} 
Suppose in addition to the assumptions of Theorem~\ref{thm1.1} that
$a$ is periodic with respect to the horizontal variables.
Then, for any $T^\ast >0$ the unique, mild solution $v$ obtained in Theorem \ref{thm1.1} can be extended to a unique, strong solution of \eqref{eq:prim} on the interval $(0,T^\ast)$.  
\end{theorem}

\section{Interpolation Inequalities for Fractional Derivatives}\label{sec:Interpolation} 

In this section we consider the Caputo fractional derivative of a function $f \in L^\infty(J;\C)$, where $J=(z_0,z_1)$ for $z_0 \in \R$ and $z_1=z_0+h$ for some $h>0$. To this end, 
we denote for $\alpha>0$ by $I^\alpha_{z_0}$ the Riemann-Liouville operator of the form
		\[
		\left(I^\alpha_{z_0}f \right)(z) 
		= \frac{1}{\Gamma(\alpha)}
		\int^z_{z_0}(z-\zeta)^{\alpha-1} f(\zeta)
		\mathrm{d}\zeta,\quad  z\in\overline{J}, 
		\]
where $\Gamma$ denotes the usual Gamma function, i.e. $\Gamma(\alpha)=\int^\infty_0 e^{-\zeta}\zeta^{\alpha-1}\mathrm{d}\zeta$.
Considering  the zero extension of $f$ to $\mathbb{R}$ still denoted by $f$ and the zero extension of $z$ from $(0,h)$ by zero onto $\mathbb{R}$ denoted by $z_+$, 
the Riemann-Liouville operator is defined as  
convolution
		\[
		I^\alpha_{z_0}f  = \frac{z^{\alpha-1}_+}{\Gamma(\alpha)} \ast f, \quad f \in L^\infty(J).
		\]
Then $I^\alpha_{z_0}f$ is the $\alpha$-times integral of $f$ from $z_0$ whenever $\alpha >0$  and $I^{\alpha_1+\alpha_2}_{z_0} = I^{\alpha_1}_{z_0}I^{\alpha_2}_{z_0}$ for all 
$\alpha_1, \alpha_2>0$, 
and we set $I^0_{z_0}f=f$.

The Caputo derivative $\partial^\alpha_z$ for $\alpha\in(0,1)$ is defined by
		\[
		\left(\partial^\alpha_z f \right)(z) := \left(I^{1-\alpha}_{z_0}(\partial_z f) \right)(z),\ z \in \overline{J},
		\]
where $\partial_z f=\partial f/\partial z$. This formula is well-defined provided  $f\in W^{1,1}(J)$ defining then $\partial^{\alpha}_z f$ as an integrable function or for  
$f\in  W^{1,\infty}(J)$ and thus in particular for Lipschitz continuous $f$. 
%
%
%
%
Indeed, by the Hausdorff-Young inequality for convolutions we have
		\begin{equation}
		\left\|\partial^\alpha_z f\right\|_{L^1(z_0,z_0+\mu)} 
		= \int^{z_0+\mu}_{z_0} \left|\partial^\alpha_z f(z)\right|\mathrm{d}z 
		\leq \frac{\mu^{1-\alpha}}{\Gamma(2-\alpha)}\|\partial_z f\|_{L^1(z_0,z_0+\mu)}
		\label{2.1}
		\end{equation}
for $\mu\in(0,h)$, since $\int^\mu_0 z^{-\alpha}\mathrm{d}z=\mu^{1-\alpha}/(1-\alpha)$.
Here we identified $\partial_z f$ with $\partial_z f \chi_{(z_0,z_0+\mu)}$ and $z_+$ with $z\chi_{(0,\mu)}$ 
denoting by $\chi_{(z_0,z_0+\mu)}$ and $\chi_{(0,\mu)}$ the corresponding characteristic functions.

We next state an interpolation inequality for $\left\|\partial^\alpha_z f\right\|_1=\left\|\partial^\alpha_z f\right\|_{L^1(J)}$.

\vspace{.2cm}\noindent
\begin{lemma} \label{lem2.1}
If  $\alpha\in(0,1)$, then 
		\begin{equation}
		\left\|\partial^\alpha_z f\right\|_1
		\leq \frac{2}{\Gamma(1-\alpha)}\|f\|^{1-\alpha}_1 \left\|\partial_z f\right\|^\alpha_1 
		\label{2.2}
		\end{equation}
holds for all $f\in W^{1,1}(J)$ satisfying $f(z_0)=0$.
\end{lemma}

\begin{proof}
We may assume that $\left\|\partial_z f\right\|_1\neq 0$ and $\|f\|_1\neq 0$. Given  $\mu\in(0,h)$ and $z\in(z_0+\mu,z_1]$, we subdivide the integral into two parts and integrate by 
parts to obtain 
		\begin{align*}
		(\partial^\alpha_z f)(z) 
		&=\frac{1}{\Gamma(1-\alpha)}\Big(
		\int^{z-\mu}_{z_0}+\int^z_{z-\mu}
		\Big)(z-\zeta)^{-\alpha} \partial_\zeta f(\zeta)\mathrm{d}\zeta 
		\\
		& = \frac{1}{\Gamma(1-\alpha)}\Big(
		\int^z_{z-\mu}(z-\zeta)^{-\alpha} \partial_\zeta f(\zeta)\mathrm{d}\zeta 
		 + \alpha \int^{z-\mu}_{z_0}(z-\zeta)^{-\alpha-1}f(\zeta)\mathrm{d}\zeta + \mu^{-\alpha} f(z-\mu)-(z-z_0)^{-\alpha}f(z_0)
		\Big).
		\end{align*}
Since $f(z_0)=0$, applying the Hausdorff-Young inequality yields
		\begin{equation}
		\int^{z_1}_{z_0+\mu}\left|\partial^\alpha_z f(z)\right|\mathrm{d}z
		\leq \frac{\mu^{1-\alpha}}{\Gamma(2-\alpha)}\|\partial_z f\|_1
		+ \frac{1}{\Gamma(1-\alpha)}\mu^{-\alpha} \left\|f\right\|_1
		+ \frac{\mu^{-\alpha}}{\Gamma(1-\alpha)} \left\|f\right\|_1.
		\label{2.3}
		\end{equation}
Combining \eqref{2.1} and \eqref{2.3}, we get
		\begin{equation}
		\left\|\partial^\alpha_z f\right\|_1
		\leq \frac{2\mu^{1-\alpha}}{\Gamma(2-\alpha)}\|\partial_z f\|_1
		+ \frac{2\mu^{-\alpha}}{\Gamma(1-\alpha)} \left\|f\right\|_1,
		\label{2.4}
		\end{equation}
and since 
by Poincar\'e's inequality
		\[
	\left\|f\right\|_1 \leq  h\|\partial_z f\|_1,
		\]
we obtain the desired estimate by setting $\mu=\|f\|_1/\|\partial_z f\|_1$ in \eqref{2.4}.
\end{proof}

\begin{remark} 
The estimate \eqref{2.2} remains valid if the $L^1$ norm is replaced by the $L^p$ norm  for some $p \in [1,\infty]$, and we obtain  in this case with essentially the same proof  
		\[
		\|\partial^\alpha_z f\|_p 
		\leq \frac{2}{\Gamma(1-\alpha)}\|f\|^{1-\alpha}_p\|\partial_z f\|^\alpha_p, 
		\]
where $\|f\|_p=\|f\|_{L^p(J)}$.
\end{remark}

We next derive an interpolation inequality for the horizontal Laplace operator in the space $L^\infty(L^1)$. 

\vspace{.2cm}\noindent
\begin{lemma} \label{lem2.3} 
Let $\alpha\in(0,1)$. Then  there exists a constant $C>0$, depending only on $\alpha$, such that
		\begin{equation}
		\left\|\nabla_H(-\Delta_H)^{-\alpha/2} f \right\|_{\infty,1}
		\leq C \|f\|^\alpha_{\infty,1} \;
		\|\nabla_H f\|^{1-\alpha}_{\infty,1}
		\label{3.5}
		\end{equation}
for all $f\in L^{\infty}(\mathbb{R}^2;L^1(J))$ with  $\nabla_H f\in L^{\infty}(\mathbb{R}^2;L^1(J))$.
%
%
\end{lemma}

%
%
%
%

\begin{proof}
We may assume that $\|\nabla_H f\|_{\infty,1}\neq 0$ and $\|f\|_{\infty,1}\neq 0$. Denoting by $G_t$ the $2$-dimensional Gauss kernel, i.e., 
$G_t(x)=(4\pi t)^{-1}\exp\left(-|x|^2/4t\right)$ for $x \in \R^2$ and $t>0$ and setting $e^{t\Delta_H}f = G_t *_H f$, where $\ast_H$ denotes convolution in the horizontal variables, only, and the negative fractional powers of $-\Delta_H$ are defined as 
\begin{equation}
		(-\Delta_H)^{-\alpha/2} f= \frac{1}{\Gamma(\alpha/2)}\int^\infty_0
		s^{\frac{\alpha}{2}-1}e^{s\Delta_H}f \mathrm{d}s,
		\label{2.6}
\end{equation}
For $\mu\in(0,\infty)$ we obtain thus 
		\[
		\lvert\nabla_H(-\Delta_H)^{-\alpha/2} f\rvert
		\leq \frac{1}{\Gamma(\alpha/2)} 
		\Big(
		\int^\infty_\mu	
		s^{\frac{\alpha}{2}-1} \lvert\nabla_H e^{s\Delta_H}f\rvert \mathrm{d}s
		+ \int^\mu_0 
		s^{\frac{\alpha}{2}-1} \lvert e^{s\Delta_H}\nabla_H f\rvert 
		\mathrm{d}s
		\Big).
		\]
We now employ the estimates
		\[
		\lvert \nabla_H e^{s\Delta_H}f\rvert
		=\lvert (\nabla_H G_s)\ast_H f\rvert
		\le \lvert \nabla_H G_s\rvert\ast_H \lvert f\rvert,
		\quad 
		\lvert e^{s\Delta_H}\nabla_H f\rvert
		\le \lvert G_s\rvert\ast_H \lvert \nabla_H f\rvert.
		\] 
By a direct calculation,
		\[
		\left| \partial_i e^{-|x|^2/4t}\right|
	 	\leq \sup_{x\in\mathbb{R}^2,t>0} \Big( \frac{|x|}{2t^{1/2}} e^{-|x|^2/8t} \Big)
		\cdot t^{-1/2}e^{-|x|^2/8t}
		= C_0 t^{-1/2}  e^{-|x|^2/8t}, \quad x \in \R^2,t>0,
		\]
with $C_0 = \sup\{ze^{-z^2/2}: z>0\} < \infty$. Thus
		
		\begin{equation}
		\left| \partial_i G_t (x) \right| 
		\leq C t^{-1/2} G_{2t}(x), \quad x \in \R^2, t>0, i=1,2.
		\label{2.7} 
		\end{equation}
and an application of Fubini's theorem yields
		\begin{align*}
		\int_J \lvert \nabla_H(-\Delta_H)^{-\alpha/2} f\rvert (\cdot,z) dz
		& \leq C_2 \int^\infty_\mu s^{\frac{\alpha}{2}-1-\frac{1}{2}}
		\lvert G_{2s}\rvert\ast_H \left(\int_J \lvert f(\cdot,z)\rvert\mathrm{d}z\right)
		\mathrm{d}s 
		\\
		&+ C_3 \int^\mu_0 s^{\frac{\alpha}{2}-1} \lvert G_s\rvert\ast_H
		\left(\int_J \lvert \nabla_H f(\cdot,z)\rvert\mathrm{d}z\right)
		\mathrm{d}s.  
		\end{align*}
Since $\lVert G_s\ast_H g\rVert_\infty\le \lVert g\rVert_\infty$ for $\partial_z g=0$, we obtain
		\begin{align*}
		\lVert \nabla_H (-\Delta_H)^{-\alpha/2}f\rVert_{\infty,1}
		 = C_4 \mu^{\frac{\alpha}{2}-\frac{1}{2}} \|f\|_{\infty,1}
		+ C_5 \mu^{\frac{\alpha}{2}} \|\nabla_H f\|_{\infty,1},
		\end{align*}
where the constants $C_j$ depends only on $\alpha$. Choosing  $\mu=\left(\|f\|_{\infty,1}/\|\nabla f\|_{\infty,1} \right)^2$, we obtained  the desired estimate.
\end{proof}

\vspace{.2cm}\noindent
\begin{remark} 
The above proof shows that  \eqref{3.5} remains true if the 
$\|\cdot\|_{\infty,1}$-norm is replaced  by the $\|\cdot\|_{\infty,p}$-norm  for any $p\in[1,\infty)$. 
\end{remark}

\vspace{.2cm}\noindent
\section{Anisotropic Estimates for the Heat Semigroup} \label{sec:anisotropic}

In this section we derive various estimates on the semigroups $e^{t\Delta_N}$ and $e^{t\Delta_{DN}}$  introduced in Section~\ref{sec:premain}. We denote by  
$e^{t\Delta_*}$ one of these semigroups and start with a regularizing decay estimate for $e^{t\Delta_*}$.

\vspace{.2cm}\noindent
\begin{lemma} \label{lem3.1} 
Given $\alpha\in(0,1)$, there exists  a constant $C>0$ such that
		\[
		\left\| e^{t\Delta_*}\partial_z I^\alpha_{z_0} f \right\|_1 
		\leq C t^{-(1-\alpha)/2}\|f\|_1, \quad t>0
		\]
		for all $f\in L^1(J)$ satisfying $I^\alpha_{z_0} f(z_1)=0$ with $z_1=z_0+h$.
\end{lemma}

\begin{proof}
Note that by duality
		\[
		\left\| e^{t\Delta_*} \partial_z I^\alpha_{z_0} f \right\|_1
		= \sup \left\{ \left\langle e^{t\Delta_*} \partial_z I^\alpha_{z_0}f,\varphi) \right\rangle
		\mid \varphi \in C^\infty_c (J), \|\varphi\|_\infty \leq 1 \right\}
		\]
where $\langle\varphi,\psi\rangle=\int_J\varphi\psi\mathrm{d}z$. Note that
		\[
		\left\langle e^{t\Delta_*} \partial_z I^\alpha_{z_0}f,\varphi \right\rangle
		=\left\langle \partial_z I^\alpha_{z_0}f, e^{t\Delta_*}\varphi \right\rangle
		=-\left\langle I^\alpha_{z_0}f, \partial_z e^{t\Delta_*} \varphi \right\rangle,
		\]
where in the last identity we used the fact that $(I^\alpha_{z_0}f)(z_1)=0$ and $(I^\alpha_{z_0}f)(z_0)=0$; the latter is trivial by definition.
Since
		\[
		\left\langle I^\alpha_{z_0}f,\psi \right\rangle
		=\left\langle f, \overline{I}^\alpha_{z_1} \psi \right\rangle
		\]
with 
		\[
		\overline{I}^\alpha_{z_1}\psi(z)
		= \frac{1}{\Gamma(\alpha)} \int^{z_1}_z (\xi-z)^{\alpha-1}\psi(\xi)\mathrm{d}\xi,
		\]
we conclude that
		\[
		\left\langle e^{t\Delta_*} \partial_z I^\alpha_{z_0}f,\varphi \right\rangle
		= - \left\langle f, \overline{I}^\alpha_{z_1} \partial_z e^{t\Delta_*} \varphi \right\rangle.
		\]
Since $\overline{I}^\alpha_{z_1} \partial_z$ resembles the Caputo derivative and $\partial_z e^{t\Delta} \varphi(z_1,t)=0$ by \eqref{3.2a} or \eqref{3.2b}, we are able to adapt 
Lemma \ref{lem2.1} to obtain
		\[
		\left\|\overline{I}^\alpha_{z_1} \partial_z e^{t\Delta_*} \varphi \right\|_\infty
		\leq \frac{2}{\Gamma(\alpha)} \left\| e^{t\Delta_*} \varphi \right\|^{\alpha}_\infty 
		\left\| \partial_z e^{t\Delta_*} \varphi \right\|^{1-\alpha}_\infty.
		\]
Notice that $\left\| e^{t\Delta_*} \varphi \right\|_\infty \leq \|\varphi\|_\infty$ for all $t>0$ and 
		\[
		\left\| \partial_z e^{t\Delta_*} \varphi \right\|_\infty
		\leq C_0 t^{-1/2} \|\varphi\|_\infty \quad\text{for all}\quad t>0
		\]
for some  constant $C_0>0$. This can be seen by extending  the problem  to the whole space problem by extending $\varphi$ periodically in a suitable way to 
$\mathbb{R}$ to obtain  $e^{t\Delta_*}\varphi=G_t*\tilde{\varphi}$, where $\tilde{\varphi}$ is the extension of $\varphi$.
Thus
		\[
		\left\| \overline{I}^\alpha_{z_1} \partial_z e^{t\Delta_*} \varphi \right\|_\infty
		\leq C_1 t^{-(1-\alpha)/2} \|\varphi\|_\infty, \quad t>0,
		\]
with $C_1$ depending on $\alpha$, only.  
We thus conclude that
		\[
		\left|\left\langle e^{t\Delta_*} \varphi_z I^\alpha_{z_0}f,\varphi \right\rangle\right|
		\leq \|f\|_1 \left\|\overline{I}^\alpha_{z_1} \partial_z e^{t\Delta_*} \varphi\right\|_\infty
		\leq C_1 t^{-(1-\alpha)/2} \|f\|_1 \|\varphi\|_\infty, \quad t>0
		\]
which gives the desired estimate.
\end{proof}

In the following, we derive further regularity estimates for the heat semigroup $e^{t\Delta}$ on $L^{\infty}(\mathbb{R}^d)$ for $d\geq 1$. We denote by 
 $R_i=\partial_i(-\Delta)^{-1/2}$ the $i$-th Riesz transform, where $\partial_i=\partial/\partial x_i$ for all  $1\leq i\leq d$.

\vspace{.2cm}\noindent
\begin{lemma} \label{lem3.2} 
Given $\alpha \in (0,1)$ and $d \in \N$, there exists a constant $C>0$ such that
\begin{enumerate}
\item[(i)] for all $x \in \mathbb{R}^d$, all $t>0$ and all $f \in L^\infty(\mathbb{R}^d)$ 
		\[
		\left|e^{t\Delta}(-\Delta)^{\alpha/2} f(x) \right| 
		\leq t^{-\alpha/2} (H_t * |f|)(x)
		\]
	where $H_t\in L^1(\mathbb{R}^d)$ such that $\lVert H_t\rVert_1\le C_\alpha$ with a constant $C_\alpha $ independent of $t>0$. In particular,
		\[
		\left\| e^{t\Delta}(-\Delta)^{\alpha/2} f \right\|_\infty
		\leq C_\alpha t^{-\alpha/2} \|f\|_\infty, \quad t>0.
		\]
	\item[(ii)] for all $x \in \mathbb{R}^d$, all $t>0$ and all $f \in L^\infty(\mathbb{R}^d)$ 
		\[
		\left| e^{t\Delta} R_i R_j (-\Delta)^{-\alpha/2} f(x) \right|
		\leq t^{-\alpha/2} (\tilde{H}_t * |f|)(x),
		\]
	with $\tilde{H}_t$ having the same properties as $H_t$.
	In particular,
		\[
		\left\| e^{t\Delta}R_i R_j (-\Delta)^{\alpha/2} f \right\|_\infty
		\leq \tilde{C}_\alpha t^{-\alpha/2} \|f\|_\infty, \quad t>0.
		\] 
	\item[(iii)] for all $x \in \mathbb{R}^d$, all $t>0$ and all $f \in L^\infty(\mathbb{R}^d)$ 
		\[
		\left| e^{t\Delta} R_i R_j \partial_k f(x) \right|
		\leq t^{-1/2} (\breve{H}_t * |f|)(x),
		\]
	with $\breve{H}_t$ having the same property as $H_t$.
	In particular,
		\[
		\left\| e^{t\Delta}R_i R_j (-\Delta)^{\alpha/2} f \right\|_\infty
		\leq \breve{C} t^{-1/2} \|f\|_\infty, \quad t>0.
		\]
	\end{enumerate}

\end{lemma}

\begin{proof}
i) Using the Bochner representation formula for fractional powers of the Laplacian (see e.g. \cite{Yoshida} p. 260)
		\begin{equation}
		(-\Delta)^{\alpha/2} f= \frac{1}{\lvert \Gamma(-\alpha/2)\rvert}\int^\infty_0
		s^{-\left(1+\alpha/2\right)}\left(e^{s\Delta} - 1\right)f \mathrm{d}s,
		\label{3.2iii}
		\end{equation}
we obtain
		\begin{align*}
		e^{t\Delta}(-\Delta)^{\alpha/2}f
		&=\frac{1}{\lvert \Gamma(-\alpha/2)\rvert}
		\int_0^\infty s^{-(1+\alpha/2)}(G_{t+s}-G_t)\ast f\,ds
		\\
		&=\frac{1}{\lvert \Gamma(-\alpha/2)\rvert}
		\int_t^\infty (s-t)^{-(1+\alpha/2)}(G_s-G_t)\ast f\,ds
		\\
		&=\frac{t^{-\alpha/2}}{\lvert \Gamma(-\alpha/2)\rvert}
		\int_1^\infty (u-1)^{-(1+\alpha/2)}(G_{tu}-G_t)\ast f\,ds, \quad t>0.
		\end{align*}
Therefore
		\[
		\lvert e^{t\Delta} (-\Delta)^{\alpha/2}f\rvert (x) 
		\le t^{-\alpha/2}(H_t\ast \lvert f\rvert)(x), \quad t>0,
		\]
	where
		\[
		H_t=\frac{1}{\lvert \Gamma(-\alpha/2)\rvert}\int_1^\infty \vert G_{tu}-G_t\vert(u-1)^{-(1+\alpha/2)}\,du
		\]
	satisfies 
		\[
		\lVert H_t\rVert_1
		\le \frac{2}{\lvert \Gamma(-\alpha/2)\rvert}\int_1^\infty (u-1)^{-(1+\alpha/2)}\,du
		=C_\alpha
		<\infty.
		\]
ii). Observe that $e^{t\Delta}R_iR_j(-\Delta)^{\alpha/2}f=\partial_i \partial_j e^{t\Delta}(-\Delta)^{-(1-\alpha/2)}$ for all $1\leq i,j \leq d$. 
The representation formula  \eqref{2.6} implies 
		\begin{align*}
		\partial_i\partial_j e^{t\Delta}(-\Delta)^{-(1-\alpha/2)}f
		&=\frac{1}{\Gamma(1-\alpha/2)}
		\int_0^\infty s^{-\alpha/2} (\partial_i\partial_j G_{t+s})\ast f\,ds
		\\
		&=\frac{1}{\Gamma(1-\alpha/2)}
		\int_t^\infty (s-t)^{-\alpha/2} (\partial_i\partial_j G_{s})\ast f\,ds, \quad t>0,
		\end{align*}
where we used the identities 
		\[
		(\partial_i \partial_j e^{t\Delta})e^{s\Delta}f
		=(\partial_i \partial_j G_t)\ast (G_s \ast f)
		=\partial_i \partial_j (G_t\ast G_s\ast f)
		=(\partial_i \partial_j G_{t+s})\ast f.
		\]
A calculation similar to the one given in the proof of Lemma \ref{lem2.3} yields 
	$\lvert \partial_i \partial_j G_t \rvert \le C t^{-1} G_{2t}$ for all $t>0$ and proceeding as above we obtain
		\[
		\lvert e^{t\Delta} R_iR_j(-\Delta)^{\alpha/2}f\rvert(x) 
		\le t^{-\alpha/2}(\tilde{H}_t\ast \lvert f\rvert)(x), \quad t>0,
		\]
	where 
		\[
		\tilde{H}_t=\frac{1}{\Gamma(1-\alpha/2)}\int_1^\infty G_{2tu}(u-1)^{-\alpha/2}u^{-1}\,du
		\]
satisfies 
$\lVert \tilde{H}_t\rVert_1=\frac{1}{\Gamma(1-\alpha/2)}\int_1^\infty (u-1)^{-\alpha/2}u^{-1}\,du =\tilde{C}_\alpha <\infty$.\\ 
iii) As above we have $e^{t\Delta}R_i R_j \partial_k f=\partial_i\partial_j\partial_k e^{t\Delta}(-\Delta)^{-1}f$ and  using the estimate 
$\lvert \partial_i \partial_j\partial_k G_t \rvert \le C t^{-3/2} G_{2t}$ for $t>0$ yields the desired result by the same arguments.
	\end{proof}

\vspace{.2cm}\noindent
\begin{remark} \label{rem3.3} 
The assertions of  Lemma \ref{lem3.1} and Lemma \ref{lem3.2} remain true for the case  $\alpha=0$ by interpreting $(-\Delta)^0$ and $I^0_{z_0}$ as identity operators. 
For  $\alpha=1$ the assertion  of Lemma \ref{lem3.1} remains true as well since then $\partial_z I_{z_0}^{\alpha} f =f$ if $f(z_0)=0$.
However, the assertion of Lemma \ref{lem3.2} is no longer true for  $\alpha=1$, since this would imply the boundedness of the Riesz transforms on  $L^{\infty}(\mathbb{R}^d)$ or $L^1(\mathbb{R}^d)$.
\end{remark}

\vspace{.2cm}\noindent
\begin{remark} \label{rem3.4} 
The assertions of Lemma \ref{lem3.1} remains valid if the $L^1$-norm is replaced by the  $L^p$-norm  for any $p \in [1,\infty]$. In fact,  we obtain
		\[
		\left\| e^{t\Delta_*}\partial_z I^\alpha_{z_0} f \right\|_p 
		\leq C t^{-(1-\alpha)/2}\|f\|_p, \quad t>0,
		\]
for all $f\in L^p(J)$ satisfying $\left( I^\alpha_{z_0} f \right)(z_1)=0$ with $z_1=z_0+h$. The proof parallels then the one given above provided we have the estimates
		\[
		\left\| e^{t\Delta_*}\varphi \right\|_p 
		\leq \|\varphi\|_p, 
		\quad
		\left\| \partial_z e^{t\Delta_*}\varphi \right\|_p 
		\leq C_0t^{-1/2}\|\varphi\|_p, \quad t>0.
		\]
at hand. These  estimates are well known and we give an elementary and self-contained proof of them  by a periodization  method, which is explained  in the next section, see 
Lemma \ref{lem4.4}. 

\end{remark}

\vspace{.2cm}\noindent
\section{Estimates for the hydrostatic Stokes semigoup}\label{sec:hydrostaticStokes} 
Consider the hydrostatic Stokes semigroup $S$ on $L^\infty_\sigma(L^p(J))$ for $p \in [1,\infty]$ as introduced in Section~\ref{sec:premain} and given by 
		\[
		S(t) = e^{t\Delta_H} \otimes e^{t\Delta_N}, \quad t>0.
		\]
This semigroup admits an extension to the space $L^\infty(L^p)$ for all $p \in [1,\infty]$, which we denote  by $S_\infty$.

The following estimates  with respect to the $\| \cdot\|_{\infty,1}$-norm for the semigroup $S_\infty$ and hence in particular for the hydrostatic Stokes semigroup $S$  will be of 
crucial importance in our iteration argument in the subsequent  Section~\ref{sec:proofs}.   

\begin{lemma} \label{lem4.1} 
Let $\alpha\in(0,1)$. Then there is a constant $C$ such that
\begin{enumerate}
\item[(i)] for all $f\in L^\infty(L^1)$ and all $t>0$ 
		\[\displaystyle 
		\left\| \nabla S_\infty(t)f\right\|_{\infty,1} 
		\leq C t^{-1/2}\|f\|_{\infty,1},
		\quad
		\left\| S_\infty(t)\nabla_H\cdot f\right\|_{\infty,1} 
		\leq C t^{-1/2}\|f\|_{\infty,1},
		\]
	\item[(i\hspace{-.1em}i)]
	for all $f\in L^\infty(L^1)$ with $I^\alpha_{z_0}f=0$ at $z=z_1$ and all $t>0$ 
		\[
		\displaystyle \left\| S_\infty(t)\partial_z I^\alpha_{z_0}f\right\|_{\infty,1} 
		\leq C t^{-(1-\alpha)/2}\|f\|_{\infty,1}.
		\]
	Moreover, for all $f\in L^\infty(L^1)$ and all $t>0$ 
		\[
		\displaystyle \left\|S_\infty(t) \partial_z f\right\|_{\infty,1} 
		\leq C t^{-1/2}\|f\|_{\infty,1}.
		\]
		%
%
%
%
%
	\item[(iii)]  
	for all $f\in L^\infty(L^1)$ and all $t>0$ 
		\[
		\left\| S_\infty(t)\PP (-\Delta_H)^{\alpha/2}f\right\|_{\infty,1}
		\leq Ct^{-\alpha/2}\|f\|_{\infty,1}.
		\]
		%
	\item[(iv)] for all $f\in L^\infty(L^1)$ and all $t>0$ 
		\[
		\left\| S_\infty(t)\PP \nabla_H\cdot f\right\|_{\infty,1}
		\leq Ct^{-1/2}\|f\|_{\infty,1}.
		\]
	\end{enumerate}
\end{lemma}
\begin{remark}\label{rem:alpha=0}
We note that assertion ii) remains true also for the cases $\alpha=0$ and $\alpha =1$ by Remark \ref{rem3.3}. The later implies that $S_{\infty}$ is a bounded semigroup in $L^{\infty}(L^1)$.
\end{remark}

\begin{proof}
i)  This assertion follows from  the estimates 
$$
\left\| \partial_z e^{t\Delta_*} f \right\|_1\leq C t^{-1/2}\|f\|_1, \quad t>0 \mbox{ and } \left\| e^{t\Delta_*} f \right\|_1 \leq \|f\|_1, \quad t>0,   
$$
and from the pointwise estimates 
		\[
		\left|\nabla_H e^{t\Delta_H}f\right|
		\leq C t^{-1/2}G_{2t}*|f|,
		\quad 
		\left|e^{t\Delta_H}f\right|\leq G_t *|f|,
		\]
compare \eqref{2.7}, as well as $e^{t\Delta_H}\partial_{x_i}f=\partial_{x_i}e^{t\Delta_H}f$ for $i=1,2$ as in the proof of Proposition \ref{thm4.5}. \\
ii) Since 
		\[
		\left| e^{t\Delta_H} g \right|(x') \leq (G_t * |g|)(x'), \quad t>0,
		\]
	Fubini's theorem implies
		\[
		\left\| e^{t\Delta_H} e^{t\Delta_N} \partial_z I^\alpha_{z_0} f(x',\cdot) \right\|_{L^1(J)}
		\leq G_t * \left\| e^{t\Delta_N} \partial_z I^\alpha_{z_0} f(x',\cdot) \right\|_{L^1(J)}, \quad t>0
		\]
for almost all $x' \in \R^2$. By Lemma \ref{lem3.1} 
		\[
		\left\| e^{t\Delta_N} \partial_z I^\alpha_{z_0} f(x',\cdot) \right\|_{L^1(J)} 
		\leq  C t^{-(1-\alpha)/2} \|f(x',\cdot)\|_{L^1(J)}, \quad t>0,
		\]
	which allows us to  conclude that
		\begin{align*}
		\left\| S_\infty(t) \partial_z I^\alpha_{z_0} f \right\|_{\infty,1} 
		 \leq  C t^{-(1-\alpha)/2} \|G_t\|_1 \|f\|_{\infty,1} 
		= C t^{-(1-\alpha)/2} \|f\|_{\infty,1}, \quad t>0. 
		\end{align*}
The proof is also  valid for the case $\alpha=0$ yielding $\left\| S_\infty(t) \partial_z f \right\|_{\infty,1} \leq C t^{-1/2}\|f\|_{\infty,1}$ for all $t>0$. \\
%
%
%
%
(iii) We verify by Lemma \ref{lem3.2} i) and ii) that 
 		\begin{align*}
 		\left\| S_\infty(t) \PP (-\Delta_H)^{\alpha/2} f(x',\cdot)\right\|_{L^1(J)}
 		&\leq \left\| e^{t\Delta_H}e^{t\Delta_N} (-\Delta_H)^{\alpha/2} f(x',\cdot) \right\|_{L^1(J)} \\ 
 		&  \quad + \sum_{1\leq i,j\leq 2} \left\| e^{t\Delta_H}e^{t\Delta_N} R_i R_j (-\Delta_H)^{\alpha/2} \overline{f} \right\|_{L^1(J)}
 		\\
 		&\leq t^{-\alpha/2 }
 		\left( \left\| H_t *_H |f|(x',\cdot) \right\|_{L^1(J)}+ h \; (\tilde{H}_t *_H \overline{f})(x')\right), \quad t>0,
	 	\end{align*}
 for almost all $x'\in \R^2$ since $\overline{f}$ is independent of $z$. By Fubini's theorem 
 		\[
 		\int_J \lvert H_t *_H |f|(x',z)\rvert \,dz 
 		= \Big(H_t *_H \int_J \lvert f(\cdot,z)\rvert\,dz \Big)(x'), \quad \mbox{a.a. } x' \in \R^2,
 		\]
	 which allows us to  conclude that
 		\begin{align*}
 		\left\| S_\infty(t) \PP (-\Delta_H)^{\alpha/2} f \right\|_{\infty,1} 
 		& \leq t^{-\alpha/2} 
 		\left( \|H_t\|_{L^1(\mathbb{R}^2)}+\|\tilde{H}_t\|_{L^1(\mathbb{R}^2)} \right)
 		\|f\|_{\infty,1}. \\
 		& \le ( C_\alpha + \tilde{C}_\alpha) t^{-\alpha/2}  \|f\|_{\infty,1}, \quad t>0.
 		\end{align*}
(iv) As above we have
		\begin{align*}
	 	\left\| S_\infty(t) \PP \nabla_H\cdot f \right\|_{L^1(J)}
	 	&\leq \left\| \nabla_H e^{t\Delta_H}   f \right\|_{L^1(J)} 
	 	+ \sum_{1\leq i,j\leq 2} 
	 	\left\| e^{t\Delta_H} R_i R_j \nabla_H\cdot \overline{f} \right\|_{L^1(J)}.
	 	\end{align*}
	 The first term was already estimated and the second one is treated in the same way  as in iii).
	
\end{proof}

In order to  extend the above  estimates to $L^\infty(L^q)$ space, it is convenient to investigate  the periodic heat semigroup.
\vspace{.2cm}\noindent
\begin{lemma} \label{lem4.2} 
Let $\mathbb{T}=\mathbb{R}/\omega_0\mathbb{Z}$ for some  $\omega_0>0$ and $f \in L^1(\mathbb{T})$. Then
		\[
		(G_t * f)(z) = \int^{\omega_0}_0 E_t (z-y) f(y)dy, \quad t>0, z \in \T, 
		\]
where $E_t(z)=\sum^\infty_{k=-\infty} G_t(z-k\omega_0)$ for $z \in \T$.  In particular, 
$$
\|G_t*f\|_{L^p(\mathbb{T})} \leq \|f\|_{L^p(\mathbb{T})}, \quad t>0.
$$If
for all $p \in [1,\infty]$. 
\end{lemma}

\begin{proof}
The above representation  for  $G_t*f$ follows by noting that 
		\[
		(G_t*f)(z) = \sum^\infty_{k=-\infty} \int^{(k+1)\omega_0}_{k\omega_0} G_t(z-y)f(y)dy, \quad t>0, z \in T, 
		\]
		and
		\[
		\int^{(k+1)\omega_0}_{k\omega_0} G_t(z-y) f(y) dy = \int^{\omega_0}_0 G_t(z-y-k\omega_0) f(y+k\omega_0) dy, \quad t>0
		\]
where  $f(y+k\omega_0)=f(y)$ for all $k \in \Z$ by periodicity. The estimate claimed follows by  Young's inequality since $\int^{\omega_0}_0 E_t(z-y) dz = \int^\infty_{-\infty} G_t(z-y) dz=1$ for 
all $t>0$ and since  $E_t\geq 0$ for all $t>0$. 
\end{proof}

\vspace{.2cm}\noindent
\begin{lemma}[Derivative estimate for the periodic heat semigroup] \label{lem4.3} 
Under the assumption of Lemma \ref{lem4.2}, there exists a constant $C>0$, independent of $\omega_0$, such that
		\[
		\left| \partial_z (G_t*f)(z) \right| \leq C t^{-1/2} \int^{\omega_0}_0 E_{2t}(z-y) \left|f(y)\right|dy, \quad  t>0, z \in \T,
		\]
In particular,
		\[
		\left\| \partial_z (G_t*f) \right\|_{L^p(\mathbb{T})} 
		\leq C t^{-1/2} \|f\|_{L^p(\mathbb{T})}, \quad t>0,
		\]
for all $p \in [1,\infty]$.
\end{lemma}

\begin{proof}
		%
By \eqref{2.7}
		\[
		\left| \partial_z G_t(z) \right| \leq C t^{-1/2} G_{2t}(z), \quad t>0, z \in \T, 
		\]
which implies the first assertion. The second one follows  by Young's inequality.  
\end{proof}

\vspace{.2cm}\noindent
\begin{lemma}[Periodization] \label{lem4.4} 
Given $1\leq p\leq\infty$, there exists a  constant $C>0$ such that
		\begin{align*}
		\left\| e^{t\Delta_*}f \right\|_{L^p(J)} \leq \|f\|_{L^p(J)} 
		\quad \mbox{and} \quad 
		\left\| \partial_z e^{t\Delta_*}f \right\|_{L^p(J)} 
		\leq C t^{-1/2} \|f\|_{L^p(J)}, \quad t>0, f \in L^p(J).
		\end{align*}

\end{lemma}

\begin{proof}
Consider the case, where  $\Delta_*=\Delta_N$. We first extend $f \in L^p(J)$ to $(z_0-h,z_0)$ 
by even extension, i.e., by setting $f(z) = f(-z)$ for $z \in (z_0-h,z_0)$. 
and extend then $f$ to  a periodic function $f_{per}$ with period $\omega_0=2h$ by $f_{per}(z) = f(z-k\omega_0)$ for $z \in \left( k\omega_0,(k+1)\omega_0 \right)$ and 
$k \in \mathbb{Z}$, see Figure~\ref{fig:extension} . It then follows that
		\[
		e^{t\Delta_N}f = \left. e^{t\Delta}f_{per} \right|_J, 
		\]
and $\|f\|_{L^p(J)}=\|f_{per}\|_{L^p(-h,h)}/2$.
The desired estimates follow then from Lemma \ref{lem4.2} and Lemma \ref{lem4.3}.

\begin{figure}
	\begin{center}
		\begin{tikzpicture}[scale=0.8]
		\draw (-0.5*pi,0) -- (2.5*pi,0) (-0.5*pi,-1) -- (-0.5*pi,1);
		\draw (-0.5*pi,-1) node[below=3pt] {$z_0-h$};
		\draw[dashed] plot[domain=-0.5*pi:0.5*pi,smooth] (\x,{-sin(\x r)});
		\draw (0.5*pi,-1) -- (0.5*pi,1);
		\draw[<->]  (0.4*pi,-0.7) -- (0.6*pi,-0.7);
		\draw (0.5*pi,-1) node[below=3pt] {$z_0$};
		
		\draw[very thick] plot[domain=0.5*pi:pi,smooth] (\x,{-sin(\x r)});
		
		\draw[very thick] plot[domain=pi:1.5*pi,smooth] (\x,{-sin(\x r)});
		\draw (1.5*pi,-1) -- (1.5*pi,1);
		\draw (1.5*pi,-1) node[below=3pt] {$z_0+h$};
		
		\draw[dashed] plot[domain=1.5*pi:2.5*pi,smooth] (\x,{-sin(\x r)});
		\draw (2.5*pi,-1) -- (2.5*pi,1);
		\draw (2.5*pi,-1) node[below=3pt] {$z_0+2h$};
		\end{tikzpicture}
		\caption{Periodic extension by even reflexion}
		\label{fig:extension}
	\end{center}
\end{figure}
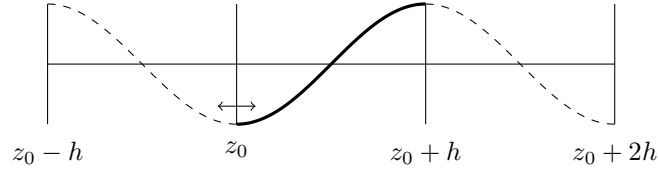

\end{proof}

These estimates are important in order  to extend Lemma \ref{lem4.1} to the situation of $L^\infty(L^q)$ spaces. 

\vspace{.2cm}\noindent
\begin{proposition} \label{thm4.5} 
Let $S_\infty$ be the semigroup on $L^\infty(\R^2;L^q(J))$ given by $S_\infty(t)=e^{t\Delta_H}\otimes e^{t\Delta_N}$. Furthermore let $\alpha\in[0,1)$ and $q\in [1,\infty]$. 
Then there exists a constant $C>0$ such that
	\begin{enumerate}
	\item[(i)] for all $f\in L^\infty(L^q)$ 
		\[
		\displaystyle \left\| \nabla S_\infty(t)f\right\|_{\infty,q} \leq C t^{-1/2}\|f\|_{\infty,q},
		\quad
		\left\| S_\infty(t)\nabla_H \cdot f\right\|_{\infty,q} \leq C t^{-1/2}\|f\|_{\infty,q}, \quad t>0.
		\]
	\item[(ii)]
	for all $f\in L^\infty(L^q)$ satisfying  $I^\alpha_{z_0}f=0$ at $z=z_1$ 
		\[
		\left\| S_\infty(t)\partial_{z} I^\alpha_{z_0}f\right\|_{\infty,q} 
		\leq C t^{-(1-\alpha)/2} \|f\|_{\infty,q}, \quad t>0.
		\]
\item[(iii)] for all $f\in L^\infty(L^q)$ 
		\[
		\left\| S_\infty(t)\PP (-\Delta_H)^{\alpha/2}f\right\|_{\infty,q} 
		\leq C t^{-\alpha/2} \|f\|_{\infty,q}, \quad t>0.
		\]
	\end{enumerate}
\end{proposition}

\begin{remark}\label{rem:l1lq}
Note that also the following $L^{\infty}(L^1)$-$L^{\infty}(L^q)$ smoothing holds for $q\in [1,\infty]$
\begin{align*}
\left\| S_\infty(t)f\right\|_{\infty,q} 
		\leq C t^{-(1-1/q)} \|f\|_{\infty,1}, \quad t>0.
\end{align*}

We also remark that, if  the case $\alpha =0$ is considered in  assertion ii), the term  $I^0_{z_0}$ is interpreted as the identity operator and there is no restriction for $f$ other 
than $f \in L^\infty(L^q)$.

\end{remark}

\begin{proof}
i) We first prove that $\left\|\partial_z S_\infty(t)f\right\|_{\infty,q} \leq Ct^{-1/2}\|f\|_{\infty,q}$ for all $t>0$. By Lemma \ref{lem4.4}
		\[
		\left\|\partial_z S_\infty(t)f(x',\cdot)\right\|_{L^q(J)} \leq C t^{-1/2} \|e^{t\Delta_H}f(x',\cdot)\|_{L^q(J)} 
		\]
for almost all $x' \in \R^2$. By Minkowski's inequality and due to the positivity of $e^{t\Delta_H}$
		\[
		\|e^{t\Delta_H}f(x',\cdot)\|_{L^q(J)} \leq  e^{t\Delta_H}\|f(x',\cdot)\|_{L^q(J)}, 
		\]
and thus 
		\begin{align*}
		\left\| \partial_z S_\infty(t)f\right\|_{\infty,q} 
		 \leq C t^{-1/2}\ \mathrm{ess.sup}_{x'}\left( e^{t\Delta_H}\|f(x',\cdot)\|_{L^q(J)} \right) \leq C t^{-1/2} \|f\|_{\infty,q}, \quad t>0.
		\end{align*}
	We next prove that $\left\| \nabla_H S_\infty(t)f \right\|_{\infty,q} \leq C t^{-1/2} \|f\|_{\infty,q}$ for all $t>0$. To this end, we estimate
		\[
		\| e^{t\Delta_N} \nabla_H e^{t\Delta_H}f(x',\cdot) \|_{L^q(J)}
		\leq \| \nabla_H e^{t\Delta_H}f(x',\cdot) \|_{L^q(J)}.
		\]
	As in Lemma \ref{lem4.3}, we observe that
		\[
		| \nabla_H e^{t\Delta_H} f(x',z)| 
		\leq C t^{-1/2}\left( G_{2t} *_H|f| \right) (x',z), 
		\]
	and applying Minkowski's inequality yields
		\[
		\| \nabla_H e^{t\Delta_H} f(x',\cdot) |_{L^q(J)}  
		\leq C t^{-1/2} (G_{2t} *_H \|f(x',\cdot)\|_{L^q(J)}).
		\]
	We thus conclude that
		\[
		\left\| \nabla_H S_\infty(t) f \right\|_{\infty,q} 
		\leq C t^{-1/2} \|f\|_{\infty,q}, \quad t>0.
		\]
The second part of i) follows from $S(t)\partial_{x_i}f=\partial_{x_i}S(t)f$ for $i=1,2$. The remaining assertions ii) and iii) follow from the $L^q$ versions of Lemma \ref{lem4.1} and 
Remark \ref{rem3.4}. 
\end{proof}

\vspace{.2cm}\noindent
\section{Proof of main results}\label{sec:proofs} 

In the following, we  construct a solution of the integral equation \eqref{5.1}.
We start by estimating the integral term for functions with vanishing vertical average. 

\vspace{.2cm}\noindent
\begin{lemma} \label{prop5.1} 
For all $\alpha \in (0,1)$ there exists a constant $C>0$ such that
\begin{align*}
		\left\|S(t)\PP\nabla\cdot(\tilde{u}\otimes v)\right\|_{\infty,1} 
		\leq C t^{-(1-\alpha)/2}
		\big(\|\nabla\tilde{v}\|_{\infty,1}\|v\|_{\infty,1}+\|\tilde{v}\|_{\infty,1}\|\nabla v\|_{\infty,1}\big)^{1-\alpha} 
		\big(\|\nabla v\|_{\infty,1}\|\nabla\tilde{v}\|_{\infty,1}\big)^\alpha. 
		\end{align*}
for all $v \in L^\infty_\sigma(L^1)$ satisfying $\overline{v}=0$ and all $\tilde{u}=(\tilde{v},\tilde{w})$ with  $\tilde{v }\in L^\infty_\sigma(L^1)$ satisfying $\overline{\tilde{v}}=0$ 
and  $\tilde{w}=\int^{z_1}_z \operatorname{div}_H \tilde{v}\;\mathrm{d}x_3$.
The statement stil holds true for $\alpha=0$.
\end{lemma}

\begin{proof}
We first note that
		\[
		\nabla\cdot(\tilde{u}\otimes v)
		=\nabla_H\cdot(\tilde{v}\otimes v)+\partial_z(\tilde{w}v).
		\]
Since $\operatorname{div}_H \overline{\tilde{v}}=0$ we obtain $\tilde{w}=0$ at $z=z_0$ and since $\tilde{w}=0$ at $z=z_1$ by definition, we see that 
$\overline{\partial_z(\tilde{w}v)}=0$. Hence,  
		\[
		\PP\nabla\cdot(\tilde{u}\otimes v)
		=\PP\nabla_H\cdot(\tilde{v}\otimes v)+\partial_z(\tilde{w}v).
		\]
The cases $\alpha = 0$ is now straight forward using Lemma~\ref{lem4.1} (ii), (iv). Consider now the case $\alpha\in(0,1)$.
Noting that $(-\Delta_H)^{(1-\alpha)/2}$,$(-\Delta_H)^{-(1-\alpha)/2}$ and $\nabla_H$ commute, we write
		\begin{align*}
		S(t)\PP\nabla\cdot(\tilde{u}\otimes v)
		& =S(t)\PP(-\Delta_H)^{(1-\alpha)/2}\nabla_H\cdot(-\Delta_H)^{-(1-\alpha)/2}(\tilde{v}\otimes v) 
		+S(t)\partial_z I^\alpha_{z_0} I^{1-\alpha}_{z_0} \partial_z (\tilde{w}v), \quad t>0, 
		\\
		& =: I + I\!I.
		\end{align*}
Applying Lemma \ref{lem4.1} iii) and Lemma \ref{lem2.3} yields 
		\[
		\|I\|_{\infty,1} 
		\leq C t^{-(1-\alpha)/2}
		\left\|\nabla_H \cdot (-\Delta_H)^{-(1-\alpha)/2}\tilde{v}\otimes 	v\right\|_{\infty,1}
		\leq C t^{-(1-\alpha)/2} \left\|\nabla_H(\tilde{v}\otimes v)\right\|^\alpha_{\infty,1}
		\left\|\tilde{v}\otimes v \right\|^{1-\alpha}_{\infty,1}, \quad t>0.
		\]
Since  $\overline{v}=0$ and $\overline{\tilde{v}}=0$, we obtain the estimates
		\begin{align*}
		\left\|\nabla(\tilde{v}\otimes v)\right\|_{\infty,1} 
		&\leq \left\|\nabla\tilde{v}\right\|_{\infty,1} \|v\|_{\infty,\infty}
		+ \|\tilde{v}\|_{\infty,\infty} \|\nabla v\|_{\infty,1}, \\
		\|\tilde{v}\otimes v \|_{\infty,1} 
		&\leq \|\tilde{v}\|_{\infty,1} \|v\|_{\infty,\infty} + \|\tilde{v}\|_{\infty,\infty} \|v\|_{\infty,1}, \\
		\|v\|_{\infty,\infty} &\leq \|\partial_z v\|_{\infty,1}, \\
		\|\tilde{v}\|_{\infty,\infty} &\leq \|\partial_z \tilde{v}\|_{\infty,1}
		\end{align*}
and the term $\|I\|_{\infty,1}$ can be thus  estimated as claimed.

In order to estimate  $\|II\|_{\infty,1}$ we observe that Lemma \ref{lem4.1} (i\hspace{-.1em}i) and Lemma \ref{lem2.1} yield 
		\[
		\|I\!I\|_{\infty,1} 
		\leq C t^{-(1-\alpha)/2}\left\|\partial^\alpha_z(\tilde{w}v)\right\|_{\infty,1} 
		\leq C t^{-(1-\alpha)/2}\|\tilde{w}v\|^{1-\alpha}_{\infty,1} 
		\left\|\partial_z(\tilde{w}v)\right\|^\alpha_{\infty,1}, \quad t>0.
		\]
Here we invoked the fact that
		\[
		I^\alpha_{z_0} \left(I^{1-\alpha}_{z_0}\partial_z(\tilde{w}v)\right)(z_1)=(\tilde{w}v)(z_1)=0.
		\]
Since
		\[
		\|\tilde{w}\|_{\infty,\infty} \leq C \left\|\partial_z\tilde{w}\right\|_{\infty,1} \leq  C \left\| \nabla_H 	\tilde{v}\right\|_{\infty,1}
		\]
we are able to estimate $\|I\!I\|_{\infty,1}$ in the same way as $I$. This completes the proof.
\end{proof}

Our next step consists in proving a similar  estimate for the above integral term, however, without assuming that the vertical average of the functions involved is vanishing.
To this end, we set 
$$
\|v\|_{1,\infty,1}:=\|v\|_{\infty,1}+\|\nabla v\|_{\infty,1}.
$$

\vspace{.2cm}\noindent
\begin{proposition} \label{prop5.2} 
There exists a constant $C>0$ such that for all $\alpha \in (0,1)$
		\begin{align*}
		\left\|S(t)\PP\nabla\cdot(\tilde{u}\otimes v)\right\|_{\infty,1} 
		& \leq C t^{-(1-\alpha)/2}\big(\|\tilde{v}\|_{1,\infty,1}\|v\|_{\infty,1}
		+\|v\|_{1,\infty,1}\|\tilde{v}\|_{\infty,1}\big)^{1-\alpha} 
		\big(\|\tilde{v}\|_{1,\infty,1}\|v\|_{1,\infty,1}\big)^\alpha. 
		\end{align*}
for all $v \in L^\infty_\sigma(L^1)$ satisfying $\nabla v \in L^\infty(L^1)$  and all $\tilde{u}=(\tilde{v},\tilde{w})$ with  $\tilde{v }\in L^\infty_\sigma(L^1)$  and 
$\nabla \tilde{v} \in L^\infty(L^1)$ and  and  $\tilde{w}= \int^{z_1}_z \operatorname{div}_H \tilde{v}\;\mathrm{d}x_3$. The statement stil holds true for $\alpha=0$.
\end{proposition}

\begin{proof}
We argue similarly as in the proof of Lemma~\ref{prop5.1}. In order to estimate $\|v\|_{\infty,\infty}$ we write  
		\[
		\|v\|_{\infty,\infty} \leq \|v-\overline{v}\|_{\infty,\infty} + \|\overline{v}\|_{\infty,\infty}.
		\]
Observing that
		\[
		\|v-\overline{v}\|_{\infty,\infty} \leq \|\partial_z v\|_{\infty,1},\quad 
		\|\overline{v}\|_{\infty,\infty} \leq \|v\|_{\infty,1}
		\]
we conclude that 
		\[
		\|v\|_{\infty,\infty} \leq \|\partial_z v\|_{\infty,1} + \|v\|_{\infty,1}.
		\]
Thus
		\begin{align*}
		\left\|\nabla(\tilde{v}\otimes v)\right\|_{\infty,1} 
		& \leq \|\tilde{v}\|_{1,\infty,1} \|v\|_{1,\infty,1} \\
		\|\tilde{v}\otimes v\|_{\infty,1} 
		& \leq \|\tilde{v}\|_{\infty,1} \|v\|_{1,\infty,1} + \|v\|_{\infty,1} \|\tilde{v}\|_{1,\infty,1}
		\end{align*}
		%
and the desired estimate follows as in the proof of Lemma \ref{prop5.1}.
\end{proof}


We now give a proof of  our main results.

\vspace{.2cm}\noindent
\begin{proof}[Proof of Theorem \ref{thm1.1}]
Consider the recursively defined sequence $(v_m)$ defined  for $t\geq 0$ by 
		\begin{align*}
		v_{m+1}(t) &:= S(t)a - \int^t_0 S(t-s)\PP\nabla \cdot (u_m(s) \otimes v_m(s))\mathrm{d}s, \quad m \in \N \\
		v_0(t) &:= S(t)a.
		\end{align*}
Applying Lemma \ref{lem4.1} (i), (ii) with $\alpha=0$, see Remark~\ref{rem:alpha=0}, we obtain
		\begin{align}\label{5.4}
		\begin{split}
		\|v_{m+1}(t)\|_{\infty,1} &
		\leq \|S(t)a\|_{\infty,1} + C_1\int^t_0(t-s)^{-1/2} \|u_m(s) \otimes v_m(s) \|_{\infty,1} \mathrm{d}s \\
		&\leq \|S(t)a\|_{\infty,1} + C_1\int^t_0 (t-s)^{-1/2} \|u_m(s)\|_{\infty,\infty} \|v_m(s)\|_{\infty,1} \mathrm{d}s  \\
		&\leq \|S(t)a\|_{\infty,1} + C_1\int^t_0 (t-s)^{-1/2} \|v_m(s)\|_{1,\infty,1} \|v_m(s)\|_{\infty,1} \mathrm{d}s,
		\end{split}
		 \end{align}
with all constants $C_1,C_2,C_3,C_4>0$ here and below being independent of $v_m$, $u_m$ and $t$. We now estimate $\|\nabla v_{m+1}(t)\|_{\infty,1}$ by Proposition \ref{prop5.2}. Since
		\[
		\nabla S(t-s) = \nabla S\big(\tfrac{t-s}{2}\big) S\big(\tfrac{t-s}{2}\big)
		\]
Lemma \ref{lem4.1} (i) and Proposition \ref{prop5.2} with $\alpha=1/2$ yield
		\begin{equation}
		\|\nabla v_{m+1}(t)\|_{\infty,1}
		\leq \|\nabla S(t)a\|_{\infty,1} 
		+ C_2\int^t_0(t-s)^{-1/2}(t-s)^{-1/4}\|v_m(s)\|^{3/2}_{1,\infty,1} \|v_m(s)\|^{1/2}_{\infty,1} \mathrm{d}s, \quad t>0. 
		\label{5.5} 
		\end{equation}
Note that in the above estimate we may also take any  $\alpha \in (0,1)$. For $m \in \N \cup \{0\}$ we now set
		\begin{align*}
		K_m(t) &:= \sup_{0<\tau<t} \tau^{1/2} \|v_m(\tau)\|_{1,\infty,1}, \\
		H_m(t) &:= \sup_{0<\tau<t} \|v_m(\tau)\|_{\infty,1}, \\
		M_m(t) &:= \sup_{0<\tau<t} \tau^{1/2} \|v_m(\tau)\|_{\infty,1}.
		\end{align*}
Estimate \eqref{5.4} combined with $\left\|S(t)a\right\|_{\infty,1}\leq\|a\|_{\infty,1}$ for all $t>0$ yields
		\begin{equation}
		H_{m+1}(t) \leq \|a\|_{\infty,1} + C_1 K_m (t) H_m(t), \quad t>0.
		\label{5.6} 
		\end{equation}
Multiplying \eqref{5.5} by $t^{1/2}$ yields
		\begin{equation}
		\sup_{0<\tau<t} \tau^{1/2} \|\nabla v_{m+1}(\tau)\|_{\infty,1}
		\leq \sup_{0<\tau<t} \tau^{1/2} \left\|\nabla S(\tau)a\right\|_{\infty,1} 
		+ C_3 K_m(t)^{3/2} H_m(t)^{1/2}, \quad t>0, 
		\label{5.7} 
		\end{equation}
and by multiplying \eqref{5.4} with $t^{1/2}$, we obtain
		\begin{equation}
		M_{m+1}(t) \leq \sup_{0<\tau<t} \tau^{1/2} \left\|S(\tau)a\right\|_{\infty,1} 
		+ C_4 t^{1/2} K_m(t) H_m(t), 
		\label{5.8} 
		\end{equation}
provided  $t \leq T$ for some $T \leq 1$. Adding \eqref{5.7} and \eqref{5.8} yields
		\begin{equation}
		K_{m+1}(t) \leq K_0(t) + C_3 K_m(t)^{3/2} H_m(t)^{1/2}
		+ C_4 t^{1/2} K_m(t) H_m(t),\quad m \geq 0
		\label{5.9} \tag{5.9}
		\end{equation}
with 
		\[
		K_0(t)
		= \sup_{0<\tau<t} \tau^{1/2} \left\|S(\tau)a\right\|_{1,\infty,1}.
		\]
If $a_1 \in BUC_\sigma(L^1)$, then
		\[
		\tau^{1/2} \left\|\nabla S(\tau)a_1\right\|_{\infty,1} \to 0 \quad\text{as}\quad 		\tau\searrow 0, \quad \hbox{and}\quad    \tau^{1/2} \left\|\nabla S(\tau)a_2\right\|_{\infty,1} \leq C \left\|a_2\right\|_{\infty,1} 
		\] 
by Lemma~\ref{lem4.1} (i).
Thus, by \eqref{5.6} and \eqref{5.9}, the sequences $(H_m)$ and $(K_m)$ fulfill the assumptions of the following Lemma \ref{lem5.4} provided  $t$ is small enough, say $t \leq t_0$ 
since $K_0(t)\to 0$ as $t\searrow  0$, 
and $\left\|a_2\right\|_{\infty,1}$ is sufficiently small. Thus the sequences $\{H_m\}$ and $\{K_m\}$ are uniformly bounded.

It is not difficult  to prove that $(v_m - S a_2)$ is a Cauchy sequence in $C\left([0,t_0],BUC_\sigma(L^1)\right)$ and that $(t^{1/2}\nabla (v_m -S a_2))$ is a Cauchy sequence in 
$C\left([0,t_0],L^\infty(L^1)\right)$. We thus obtain  $v$ as the limit of $(v_m)$ satisfying the desired estimate. 

The proof of the uniqueness follows a similar line of arguments. Let $v,\tilde{v}$ be two solutions, then 
	\begin{align*}
	(v - \tilde{v})(t) = \int^t_0 S(t-s)\PP\nabla \cdot \left(u(s) \otimes (v-\tilde{v})(s) + (u-\tilde{u}) \otimes \tilde{v}(s)\right) \mathrm{d}s, \quad t>0,
	\end{align*}
and one obtains as above using Lemma~\ref{lem4.1} (i) and 
 Proposition \ref{prop5.2} with $\alpha=1/2$,
and setting 
\begin{align*}
K(v)(t):= \sup_{0<\tau<t} \tau^{1/2} \|v(\tau)\|_{1,\infty,1} \quad \hbox{and} \quad
H(v)(t):= \sup_{0<\tau<t} \|v(\tau)\|_{\infty,1}
\end{align*}
that for $N(v)(t):= \max\{K(v)(t), H(v)(t)\}$ one has
	\begin{align*}
	N(v) &\leq C (K(v)H(v-\tilde{v}) + K(v-\tilde{v})H(v))^{1/2}(K(v)K(v-\tilde{v}))^{1/2} \\ &+ C(K(\tilde{v})H(v-\tilde{v}) 
	+  K(v-\tilde{v})H(\tilde{v}))^{1/2}(K(\tilde{v})K(v-\tilde{v}))^{1/2}.
	\end{align*}
Hence one obtains
	\begin{align}\label{eq:N}
N(v - \tilde{v}) &\leq N(v-\tilde{v})  C \left\{(K(v) + H(v))^{1/2}K(v)^{1/2} 
+ (K(\tilde{v}) +  H(\tilde{v}))^{1/2}K(\tilde{v})^{1/2}\right\}.
	\end{align}
By assumption, if $t$ is small we have
\begin{align*}
K(\tilde{v})(t), K(v)(t)\leq C \left\|a_2\right\|_{\infty,1}, \quad \hbox{and} \quad 
H(\tilde{v}), H(v)(t)\leq C \left\|a\right\|_{\infty,1}. 
\end{align*}
Thus supposing that $t$ and $\left\|a_2\right\|_{\infty,1}\cdot \left\|a\right\|_{\infty,1}$ are small enough, one has 
	\begin{align*}
C \left\{(K(v) + H(v))K(v) 
+ (K(\tilde{v}) +  H(\tilde{v}))K(\tilde{v})\right\} <1,
	\end{align*}
and therefore  by \eqref{eq:N} one has $K(v - \tilde{v})=0$ on $(0,T_0)$ and $H(v - \tilde{v})=0$ on $[0,T_0]$ for some $0<T_0\leq T$. Iterating this argument it follows that the solutions are unique on $[0,T]$.


\end{proof}

\vspace{.2cm}\noindent
\begin{lemma} \label{lem5.4} 
Let $A,\ve > 0$ be constants and assume that $\{H_m\} \subset \R$ and  $\{K_m\} \subset \R$ are sequences satisfying 
\begin{align*}
H_0    &\leq A,  \qquad H_{m+1} \leq A + C_1 H_m K_m,\\
K_0 &\leq \varepsilon, \qquad
K_{m+1} \leq \varepsilon + C_2 K^{3/2}_m H^{1/2}_m + (4A)^{-1} K_m H_m,
\end{align*}
for all $m\geq 0$, where $C_1>0$ and $C_2>0$ are constants independent of $m$. Then there exists $\varepsilon_0 = \varepsilon_0(C_1,C_2,A)>0$ such that $\{K_m\}$ and $\{H_m\}$ 
are bounded sequences provided that $\varepsilon\leq\varepsilon_0$.
\end{lemma}

\begin{proof}
Note first that if $K_m\leq 1/(2C_1)$ for $m\leq m_0$, then $H_m\leq 2A$ for all $m\leq m_0+1$. Next, we choose $\varepsilon$ small enough so that the graphs of $y=x$ and 
$y=\varepsilon+\sqrt{2A}C_2x^{3/2}+x/2$ have an intersection. Denote by  $x_0(\varepsilon)$ the abscissa of the intersection point closest to $x=0$. 
Clearly $x_0(\varepsilon)\downarrow 0$ as $\varepsilon\to 0$.

Choose now  $\varepsilon_0$ so small that $x_0(\varepsilon_0)<1/(2C_1)$. Then, $K_m\leq x_0(\varepsilon)$ and $H_m \leq 2A$ for all $m \geq 1$ provided $\ve \leq \ve_0$.  
Indeed, we proved this by induction. The estimate is trivial for $m=1$. Assume that $K_m\leq x_0(\varepsilon),\ H_m\leq 2A$ for all $m\leq m_0$. Since 
$x_0(\varepsilon)<1/(2C_1)$, the inequality for $H_m$ implies $H_{m+1}\leq 2A$ and the inequality for $K_m$ implies $K_{m+1}\leq x_0(\varepsilon)$ by the choice of 
$x_0(\varepsilon)$ since $H_m\leq 2A$. We thus conclude that $K_m \leq x_0(\varepsilon)$ and $H_m\leq 2A$.
\end{proof}
%

%

The solution $v$ constructed in Theorem \ref{thm1.1} exists at least for some nontrivial time interval $[0,T)$, where $T>0$ depends on $a$. Given  $a\in BUC_\sigma(L^p)$ for some $p>1$ we 
are unfortunately unable to estimate $T$ from below by terms involving the norm of $a$, only.  However, in Proposition~\ref{cor1} we claim that 
 $v \in C\left( [0,T),BUC_\sigma(L^p) \right)$ for all $p \in (1,\infty)$ in the same time interval.
%
%
%
%

\vspace{.2cm}\noindent
\begin{proof}[Proof of Proposition~\ref{cor1}]
We estimate the integral equation \eqref{5.1} by writting $S(t)=S(\tfrac{t}{2})S(\tfrac{t}{2})$ and then using the $L^p$-$L^1$-estimate from Remark~\ref{rem:l1lq} and Proposition~\ref{prop5.2} with $\alpha=0$ to obtain 
		\begin{align*}
		\|v(t)\|_{\infty,p} 
		& \leq \|S(t)a\|_{\infty,p} 
		+ C \int^t_0 (t-s)^{-(1-1/p)/2} (t-s)^{-1/2} \|v(s)\|_{1,\infty,1} \|v(s)\|_{\infty,1} \mathrm{d}s, 
		\\
		& \leq \|S(t)a\|_{\infty,p} 
		+ C \int^t_0(t-s)^{-(1-1/2p)} s^{-1/2} s^{1/2}\|v(s)\|_{1,\infty,1} \|v(s)\|_{\infty,1} \mathrm{d}s, \quad t>0. 
		\end{align*}
Since $a \in L^\infty_\sigma(L^p)$ and $v \in L^\infty\left( 0,T; L^\infty_\sigma(L^1) \right)$ by Theorem \ref{thm1.1} we see that $t^{1/2-1/2p} v$ is in $L^\infty\left( 0,T; L^\infty_\sigma(L^p) \right)$.
%

First note that
		\[
		\nabla S(t) = \nabla S\big(\tfrac{t}{3}\big)S\big(\tfrac{t}{3}\big)S\big(\tfrac{t}{3}\big), \quad t>0.
		\]
Differentiating \eqref{5.1}, applying first Proposition \ref{thm4.5} (i), second using the $L^p$-$L^1$-estimate from Remark~\ref{rem:l1lq} and third, applying Proposition~\ref{prop5.2} with $\alpha\in (0,1)$ arbitrary yields 
		\[
		\|\nabla v(t)\|_{\infty,p} \leq \|\nabla S(t)a\|_{\infty,p}
		+ C \int^t_0 (t-s)^{-(1-1/2p+(1-\alpha)/2)} s^{-(1+\alpha)/2} s^{(1+\alpha)/2}\|v(s)\|^{1+\alpha}_{1,\infty,1} \|v(s)\|^{1-\alpha}_{\infty,1} \mathrm{d}s.
		\]
This yields the desired bound for $t^{1-2/p}\nabla v$. The continuity of $v$ follows from strong continuity of $S$.
\end{proof}

%
%

\vspace{.2cm}\noindent
\begin{proof}[Proof of Proposition~\ref{cor2}]
We argue similarly as in the proof of Theorem \ref{thm1.1}. Setting
		\[
		L_m (t) := \sup_{0<\tau<t} \tau^\mu ||v_m(\tau)\|_{1,\infty,1}, \quad 0<t<T,
		\]
we obtain by \eqref{5.4} for $m\ge 0$ and $t\in (0,T)$
		\[
		H_{m+1}(t) \leq \|a\|_{\infty,1}+C_1 t^{1/2-\mu} L_m(t)H_m(t)
		\]
instead of \eqref{5.6}.
Similarly, instead of \eqref{5.9}, we now obtain 
		\[
		L_{m+1}(t) 
		\leq \|a\|_{\infty,1}
		+[a]_\mu
		+C_3 t^{(1/2-\mu)/2} L^{3/2}_m(t)H_m(t)^{1/2}
		+C_4 t^{1/2} L_m(t)H_m(t).
		\]
		%
%
%
%
%
It follows that if $T$ fulfills $1/T \geq  \min\left(c_*\vertiii{a},1\right)^{2/(1/2-\mu)}$ 
for  some $c_*$ independent of $a$, then $(L_m)$ and $(H_m)$ are bounded sequences for $t \in [0,T]$. Moreover, $\{v_m\}$ is a Cauchy sequence in $C\left([0,T],BUC_\sigma(L^1)\right)$, which 
is proved  as before.
\end{proof}

Since $e^{t\Delta_H}$ as well as $\PP$ and the nonlinearity leave horizontal periodicity invariant we obtain the following.
\begin{lemma}\label{lemma:per}
If $a$ is in addition to the assumption of Theorem~\ref{thm1.1} periodic with respect to the horizontal variables, so is the solution $v(t)$ for $t>0$.
\end{lemma}

\begin{proof}[Proof of Theorem \ref{thm1.2}]
In order to extend the local solutions to a global one we make use of the  regularization of the solution. By Theorem~\ref{thm1.1} $v(t_0),\nabla v(t_0)\in BUC(L^1)$ for $t_0>0$, in particluar $v(t_0)\in BUC(W^{1,1})$, and since $W^{1,1}(J)\hookrightarrow L^p(J)$ for all $p\in [1,\infty]$, one has $v(t_0)\in BUC(L^{p})$. Now, by Proposition~\ref{cor1} $v(t_1),\nabla v(t_1)\in BUC(L^p)$ for $t_1>t_0$, and in particular, using Lemma~\ref{lemma:per}, where w.l.o.g. the period is $1$, one has for the restriction
$$v(t_1)\vert_{[0,1]^2\times J}\in \{v\in W^{1,p}([0,1]^2\times J)\mid v \hbox{ periodic in } x,y,  \div_H \overline{v}=0\} \quad 0<t_1<T.$$ 
Using this for $p\geq 2$ as new initial value, it follows using e.g. \cite{GigaGriesHieberHusseinKashiwabara2017} that $v$ extends to a global strong solution. 

%

\end{proof}


\begin{thebibliography}{15}



\bibitem{Chemin}
H.~Bahouri, J.-Y.~Chemin and R.~Danchin.
\newblock {\em Fourier analysis and nonlinear partial differential equations.}
\newblock Grundlehren der Mathematischen Wissenschaften, Springer, Heidelberg, 2011.
\newblock \doi{10.1007/978-3-642-16830-7}

\bibitem{CT07}
Ch.~Cao and E.~Titi.
\newblock Global well--posedness of the three-dimensional viscous primitive equations of large scale ocean and atmosphere dynamics.
\newblock {\em Annals of Mathematics}, 166:245--267, 2007.
\newblock \doi{10.4007/annals.2007.166.245}




\bibitem{GigaSaal}
M.-H.~Giga, Y.~Giga and J.~Saal.
\newblock {\em Nonlinear partial differential equations - Asymptotic behavior of solutions and self-similar solutions.}
\newblock Progress in Nonlinear Differential Equations and their Applications CRC Press, Birkh\"auser Boston, Inc., Boston, MA, 2010.
\newblock \doi{10.1007/978-0-8176-4651-6}


\bibitem{Gig86}
Y.~Giga.
\newblock Solutions for semilinear parabolic eqautions in $L^p$ and regularity of weak solutions of the Navier-Stokes system. 
\newblock {\em J. Differential Equations.}, 62:186--212, 1986.
\newblock \doi{10.1016/0022-0396(86)90096-3}


\bibitem{GigaGriesHusseinHieberKashiwabara2016}
Y.~Giga, M.~Gries, A.~Hussein, M.~Hieber and T.~Kashiwabara.
\newblock Bounded $H^{\infty}$-Calculus for the hydrostatic Stokes operator on $L^p$-spaces and applications.             
\newblock {\em Proc. Amer. Math. Soc.}, 145(9):3865--3876, 2017.
\newblock \doi{10.1090/proc/13676}


\bibitem{GigaGriesHieberHusseinKashiwabara2017}
Y.~Giga, M.~Gries, A.~Hussein, M.~Hieber and T.~Kashiwabara.
\newblock Analyticity of solutions to the primitive equations.
\newblock In Preparation.

\bibitem{Guillenetall2001}
F.~Guill\'en-Gonz\'alez, N.~Masmoudi and M.~Rodr\'iguez-Bellido.
\newblock Anisotropic estimates and strong solutions of the primitive equations.
\newblock {\em Differential Integral Equations}, 14(11):1381--1408, 2001.
\newblock \url{https://projecteuclid.org/euclid.die/1356123030}

\bibitem{HieberHusseinKashiwabara2016}
M.~Hieber, A. Hussein, and  T.~Kashiwabara.
\newblock Global Strong $L^p$ Well-Posedness of the 3D Primitive Equations with Heat and Salinity Diffusion.           
\newblock {\em J. Differential Equations}, 261:6950-6981, 2016.
\newblock \doi{10.1016/j.jde.2016.09.010}

\bibitem{HiKa16}
M.~Hieber and T.~Kashiwabara.
\newblock Global Strong Well--Posedness of the Three Dimensional Primitive Equations in $L^p$--Spaces.             
\newblock {\em Arch. Rational Mech. Anal.}, 221(3): 1077--1115, 2016.   
\newblock \doi{10.1016/j.jde.2016.09.010}

\bibitem{Kat84}
T. Kato. 
\newblock Strong $L^p$-solutions of the Navier-Stokes equations in $\R^m$, with applications to weak solutions
 \newblock {\em Math. Z.}, 187:471-480, 1984.
\newblock \doi{10.1007/BF01174182}

\bibitem{Kobelkov2007}
G.~M.~Kobelkov.
\newblock Existence of a solution ``in the large'' for ocean dynamics equations.
\newblock {\em J. Math. Fluid Mech.}, 9(4):588--610, 2007.
\newblock \doi{10.1007/s00021-006-0228-4}

\bibitem{KPRZ14}
I.~Kukavica, Y.~Pei, W.~Rusin and M.~Ziane.
\newblock Primitive equations with continuous initial data.
\newblock {\em Nonlinearity}, 27:1135--1155, 2014.
\newblock \doi{10.1088/0951-7715/27/6/1135}

\bibitem{KuZi07}
I. Kukavica and M. Ziane.
\newblock On the regularity of the primitive equations of the ocean.
\newblock {\em Nonlinearity}, 20:2739-2753, 2007. 
\newblock \doi{10.1088/0951-7715/27/6/1135}
	


\bibitem{Gilles}
P.~G.~Lemari\'{e}-Rieusset.
\newblock {\em The {N}avier-{S}tokes problem in the 21st century.}
\newblock CRC Press, Boca Raton, FL, 2016.
\newblock \doi{10.1201/b19556}	
	
\bibitem{LiTi17}
J.~Li and E.~Titi.
\newblock  Existence and uniqueness of weak solutions to viscous primitive equations for certain class of discontinuous initial data.
\newblock {\em SIAM J. Math. Anal.}, 49(1):1--28, 2017.
\newblock \doi{10.1137/15M1050513}

\bibitem{LiTiti2016}
J.~Li and E.~Titi.
\newblock Recent Advances Concerning Certain Class of Geophysical Flows.
\newblock In {\em Handbook of Mathematical Analysis in Mechanics of Viscous Fluids}, Springer International Publishing, 2016.
\newblock \doi{10.1007/978-3-319-10151-4_22-1}



\bibitem{Lionsetall1992}
J.~L.~Lions, R.~Temam and Sh.~H.~Wang.
\newblock New formulations of the primitive equations of atmosphere and applications.
\newblock {\em Nonlinearity}, 5(2):237--288, 1992.
\newblock \url{http://stacks.iop.org/0951-7715/5/237}

\bibitem{Lionsetall1992_b}
J.~L.~Lions, R.~Temam and Sh.~H.~Wang.
\newblock On the equations of the large-scale ocean.
\newblock {\em Nonlinearity}, 5(5):1007--1053, 1992.
\newblock \url{http://stacks.iop.org/0951-7715/5/1007}

\bibitem{Lionsetall1993}
J.~L.~Lions, R.~Temam and Sh.~H.~Wang.
\newblock Models for the coupled atmosphere and ocean. ({CAO} {I},{II}).
\newblock {\em Comput. Mech. Adv.}, 1:3--119, 1993.

\bibitem{Yoshida}
K. Yosida. 
\newblock {\em Functional Analysis}.
\newblock Springer, Classics in Mathematics, 1995.
\newblock \doi{10.1007/978-3-642-61859-8}



\end{thebibliography}
\end{document}